\documentclass[11pt]{amsart}
\usepackage{amsmath}
\usepackage[dvips]{epsfig}
\theoremstyle{plain}
\newtheorem{theorem}{Theorem}[section]
\newtheorem{lemma}[theorem]{Lemma}

\theoremstyle{definition}
\newtheorem{remark}[theorem]{Remark}

\numberwithin{equation}{section}

\def\al{\aligned}
\def\eal{\endaligned}
\def\be{\begin{equation}}
\def\ee{\end{equation}}

\begin{document}

\title[Upper Bound for Hessian Matrices]{An Upper Bound for Hessian Matrices
\\of Positive Solutions of Heat Equations}
\author{Qing Han}
\address{Department of Mathematics\\
University of Notre Dame\\
Notre Dame, IN 46556} \email{qhan@nd.edu}
\author{Qi S. Zhang}
\address{Department of Mathematics\\
University of California, Riverside\\
Riverside, CA 92521}  \email{qizhang@math.ucr.edu}
\begin{abstract}
We prove global and local upper bounds for the Hessian of log
positive solutions of the heat equation on a Riemannian manifold.
The metric is either fixed or evolves under the Ricci flow.
These upper bounds supplement the well-known global lower bound.
\end{abstract}\thanks{The first author acknowledges the support of NSF
Grant DMS-1105321. }
\maketitle

\section{Introduction}

Gradient bounds for log solutions of the heat equation have
appeared in the important papers by Li and Yau \cite{Li-Yau1986}
and by Hamilton \cite{Hamilton1993}. The main result in
\cite{Li-Yau1986} can be regarded as a lower bound of the
Laplacian of log solutions under the assumption that the Ricci
curvature is bounded from below. These bounds came in both global
and local versions. The main result in \cite{Hamilton1993} is a
global lower bound for the Hessian matrices of log solutions under
certain curvature assumptions.  Many applications of these results
have been found in numerous situations. See for example the papers \cite{BCP:1},
\cite{CN:1}, \cite{Ca:1}, \cite{CaH:1},  \cite{CH:1},
\cite{Hamilton1995}, \cite{LiTian1995}, \cite{LNVV:1}, \cite{Ni:1},
\cite{P:1} and also the books
\cite{L:1} and \cite{SY:1}.

In this paper, we derive an upper
bound of the Hessian matrices of log solutions of the heat equation
on manifolds.
We prove both a global and
local version of the bounds which take two different forms and
which are generally sharp in respective cases. While the global
version can be proven by building on the ideas in
\cite{Hamilton1993}, the local version requires additional
quantities and calculations and appears to be unknown even in
the Euclidean setting. We give a unified proof for both
versions.
In addition, we generalize the Hessian bound for log
solutions to the Ricci flow case. Interestingly, the Ricci flow
induces a cancelation effect which makes the curvature assumption
less restrictive than the fixed metric case.

In order to present results, let us fix some notations to be
used throughout the paper. We denote by $M$ a  Riemannian
manifold with metric $g$. For simplicity of presentation, we
always assume that $M$ is a compact manifold without boundary in
this paper. The local bounds clearly hold without this assumption.
The global bounds also hold when one imposes suitable conditions
at infinity. We use $Ric$ and $Rm$ to denote the Ricci and full
curvature tensor respectively. In local coordinates, the metric
$g$ is denoted by $g_{ij}$, the Ricci curvature by
$R_{ij}$ etc and the curvature tensor by $R_{ijkl}$ etc. We use
the convention that $R_{ij} = g^{kl} R_{iklj}$ in local
coordinates. The
Hessian of a function $u$ is written as $u_{ij}$.
If $V$ is a 2-form on $M$ and $\xi$ is a vector field, then in local
coordinates, we use $\xi^T V \xi$ to denote $V(\xi, \xi)$.
The distance function is denoted by $d(x, y)$ in the
fixed metric case and by $d(x, y, t)$ in the Ricci flow case.
A geodesic ball is
denoted by $B(x, r)$ or $B(x, r, t)$ where $x$ is a point in $M$ and
$r$ is the radius, and $t$ is the time when the metric changes.
For $R>0$ and $T>0$, a parabolic cube is defined by
\[
Q_{R,T}(x_0,t_0) =\{(x, t) \, | \, d(x_0, x, t) < R, \, t_0-T < t
\le t_0 \}.
\]
For the fixed metric, we simply have
$$Q_{R,T}(x_0,t_0) = B(x_{0}, R) \times (t_{0} - T, t_{0}].$$
A positive constant is denoted by $c$ and $C$ with or without index, which may change
from line to line.

We first state our upper bound of the Hessian matrix of positive
solutions for a fixed metric.

\begin{theorem}\label{ThmHessfixed} Let $M$ be a Riemannian manifold
with a metric $g$.

$\operatorname{(a)}$ Suppose $u$ is a solution of
$$\partial_{t} u- \Delta u =0\quad \text{in} \ M \times (0, T].$$
Assume $0 < u \leq A$. Then, $$ t (u_{ij}) \leq
u(5+Bt)\left( 1+ \log \frac{A}{u}\right)\quad \text{in} \ M \times (0, T],$$
where $B$
is a nonnegative constant depending only on the
$L^\infty$-norms of curvature tensors
and the gradient of Ricci curvatures.

$\operatorname{(b)}$
Suppose $u$ is a
solution of
$$\partial_{t} u- \Delta u = 0\quad\text{in }Q_{R,T}(x_0,t_0).$$
Assume $0 < u \leq A$. Then,
$$(u_{ij})\leq C u \left( \frac{1}{T}
+ \frac{1}{R^{2}} + B\right) \left(1+\log \frac{A}{u}\right)^{2}
\quad \text{in} \ Q_{\frac{R}{2}, \frac{T}{2}}(x_0,t_0).$$ where
$C$ is a universal constant and $B$ is a nonnegative constant
depending only on the $L^\infty$-norms of curvature tensors and
the gradient of Ricci curvatures in $Q_{R,T}(x_0,t_0)$.
\end{theorem}

In Theorem \ref{ThmHessfixed}, Part (a) is
the global version and Part (b) the local version. We point out that
different powers of $1+\log(A/u)$ appear in the right-hand sides. We will
illustrate by an example that the extra power in the local version is
optimal. This phenomenon adds more variety to the long list of differential
Harnack inequalities.

We remark that the global estimate in Part (a) can be obtained
as a consequence of the lower
Hessian estimate and the upper Laplace estimate by Hamilton
\cite{Hamilton1993}. Arguments in this paper work for both global and local
estimates, and they extend to the Ricci flow case.

Next, we state our upper bound for the heat equation coupled with the
Ricci flow. It has a similar form to the fixed metric case.

\begin{theorem}\label{Thm-LogHessRic} Let $M$ be a Riemannian manifold
with a family of metrics $g=g(t)$ satisfying
$$\partial_tg=-2Ric\quad\text{in }M\times (0,T].$$

$\operatorname{(a)}$ Suppose $u$ is a solution of
$$\partial_{t} u- \Delta u =0\quad \text{in} \ M \times (0, T].$$
Assume $0 < u \leq A$. Then, $$ t (u_{ij}) \leq  u (18+Bt)\left( 1+ \log \frac{A}{u}\right)\quad \text{in} \ M \times (0, T],$$
where $B$
is a nonnegative constant depending only on the $L^\infty$-norms of curvature tensors.

$\operatorname{(b)}$
Suppose $u$ is a
solution of
$$\partial_{t} u- \Delta u = 0\quad\text{in }Q_{R,T}(x_0,t_0).$$
Assume $0 < u
\leq A$. Then,
$$ (u_{ij}) \leq C u \left( \frac{1}{T} + \frac{1}{R^{2}} + B\right) \left(1+\log \frac{A}{u}\right)^{2}
\quad \text{in} \ Q_{\frac{R}{2}, \frac{T}{2}}(x_0,t_0).$$ where
$C$ is a universal constant and $B$ is a nonnegative constant
depending only on the $L^\infty$-norms of curvature tensors in
$Q_{R,T}(x_0,t_0)$.
\end{theorem}

We point out that the constant $B$ in Theorem \ref{Thm-LogHessRic}
does not depend on the gradient of Ricci curvatures.

We now describe briefly the method to prove both results. As
expected, the general idea of the proof is still the Bernstein
technique of finding a quantity (auxiliary function) involving the
Hessian, which satisfies a nonlinear differential equation
amenable to the maximum principle. The work is to find such a
quantity.  In this paper, the central quantity is the quotient of
the Hessian of the solution with the density of Boltzmannn entropy
of the solution,  i.e., $\frac{ u_{ij}}{ u \ln u}$, where $u$ is a
positive solution of the heat equation. Another quantity is
$\frac{|\nabla u|}{ u \ln u}$. A Bernstein type argument is
carried out for a combination of these two quantities and suitable
cutoff functions as in \cite{Li-Yau1986}.

The rest of the paper is organized as follows. In Section
\ref{sec-FixedMetric},
we  prove Theorem \ref{ThmHessfixed} about the Hessian
bound on a Riemannian manifold with a fixed metric. In Section
\ref{sec-RicciFlow}, we
prove Theorem \ref{Thm-LogHessRic}  which treats the heat equation
coupled with the Ricci flow.

{\it Acknowledgment.} We wish to thank Professor Lei Ni for helpful
suggestions.

\section{Heat Equations under the Fixed Metric}\label{sec-FixedMetric}

Let $M$ be a Riemannian manifold with metric $g$
and $\Delta$ be the Laplace-Beltrami operator.
We consider a positive solution $u$ of the heat equation
$$
u_{t} = \Delta u\quad\text{in } M \times (0, \infty).$$
We
assume
$$0<u\le A.$$ Set
$$ f= \log \frac uA.$$ Let $\{x^1, ..., x^n\}$ be a local orthonormal frame
at a point, say $p \in M$.
Then
$$ f_{i} = \frac{u_{i}}{u }, \quad f_{ij} = \frac{u_{ij}} {u} -\frac{u_{i}u_{j}}{u^{2}},$$
and hence
$$f_{t} = \Delta f+ |\nabla f|^{2}.$$

We first derive two  equations on which the theorems are based.

\begin{lemma}
\label{Lemma-v(ij)} Set
$$ v_{ij} = \frac{u_{ij}}{ u(1-f)}.$$ Then,
$$
\begin{aligned}
 &\left(- \partial_{t}+\Delta-\frac {2f}{1-f} \nabla
f\cdot\nabla\right) v_{ij} = \frac{|\nabla f|^{2}}{1-f} v_{ij} \\
&\qquad+ \frac{1}{u (1-f)} [- 2
 R_{kijl} u_{kl} + R_{il} u_{jl} + R_{jl} u_{il} + ( \nabla_i R_{jl} +
\nabla_j R_{il} - \nabla_l R_{ij}) u_l ].
\end{aligned}
$$
\end{lemma}

\begin{proof} By noting
$$\partial_t\big(u(1-f)\big)=-u_tf,\quad \partial_k\big(u(1-f)\big)=-u_kf,$$
we have
\begin{align*}
\partial_{t} v_{ij} &=  \frac{u_{ijt}}{u (1-f)} + \frac{u_{ij} f u_{t}}{u^2 (1-f)^{2}},\\
\partial_{k} v_{ij} &=  \frac{u_{ijk}}{u (1-f)} + \frac{u_{ij} f u_{k}}{u^2 (1-f)^{2}}.\end{align*}
Recall the commutation formula (see \cite{CH:1} p219 e.g.): if
$\Delta u - \partial_t u = 0$, then the Hessian $u_{ij}$ satisfies
$$
-\partial_t u_{ij} + \Delta u_{ij} = - 2
 R_{kijl} u_{kl} + R_{il} u_{jl} + R_{jl} u_{il} + ( \nabla_i R_{jl} +
\nabla_j R_{il} - \nabla_l R_{ij}) u_l.$$
Note
$$\Delta v_{ij} =
 \frac{\Delta u_{ij}}{u (1-f)} +
 \frac{u_{ij} f \Delta u}{u^2 (1-f)^{2}}+\frac{2u_{ijk}u_kf}{u^2 (1-f)^2} + \frac{u_{ij} f_k u_{k}}{u^2 (1-f)^{2}}+ \frac{2u_{ij} f^2 u_{k}^2}{u^3 (1-f)^{3}}.
$$
With $u_k=uf_k$ and the commutation formula,
 we then have
\begin{align*} (\Delta -\partial_{t}) v_{ij}
&= \frac{2 f u_{ijk} f_{k}}{u (1-f)^{2}} +
\frac {u_{ij} |\nabla f|^{2}}{u (1-f)^{2}} +
\frac{2 u_{ij}}{u (1-f)^{3}} f^{2} |\nabla f|^{2} \\
&\qquad + \frac{1}{u (1-f)} [- 2
 R_{kijl} u_{kl} + R_{il} u_{jl} + R_{jl} u_{il} + ( \nabla_i R_{jl} +
\nabla_j R_{il} - \nabla_l R_{ij}) u_l ]\\
&=\frac{2 f f_{k}}{1-f} \left[\frac {u_{ijk}}{u (1-f)}+ \frac{u_{ij}
ff_{k}}{u (1-f)^{2}}\right] + \frac{u_{ij}}{u (1-f)} \frac{|\nabla
f|^{2}}{(1-f)}\\
&\qquad +\frac{1}{u (1-f)} [- 2
 R_{kijl} u_{kl} + R_{il} u_{jl} + R_{jl} u_{il} + ( \nabla_i R_{jl} +
\nabla_j R_{il} - \nabla_l R_{ij}) u_l ].
\end{align*}
With the help of the expression for $\partial_kv_{ij}$,
 we obtain
\begin{align*} (\Delta- \partial_{t}) v_{ij} &= \frac{2f}{1-f} f_{k}
\partial_{k} v_{ij} + v_{ij} \frac{|\nabla f|^{2}}{(1-f)}\\
&\qquad +\frac{1}{u (1-f)} [- 2
 R_{kijl} u_{kl} + R_{il} u_{jl} + R_{jl} u_{il} + ( \nabla_i R_{jl} +
\nabla_j R_{il} - \nabla_l R_{ij}) u_l ].
\end{align*} This is the
desired result.
\end{proof}

\begin{lemma}\label{Lemma-w(ij)} Set
$$ w_{ij} = \frac{u_{i} u_{j}}{u^2(1-f)^{2}}.$$ Then,
\[
\al
&\left(- \partial_{t}+\Delta-\frac {2f}{1-f} \nabla f\cdot\nabla \right) w_{ij} \\
&\qquad =  \frac{2 |\nabla f|^{2}}{1-f} w_{ij}
+  2(v_{ki} + f w_{ki})(v_{kj} +f w_{kj}) +  R_{ik} w_{kj} + R_{jk} w_{ki}.
\eal
\]
\end{lemma}

\begin{proof} We proceed similarly as in the proof of
Lemma \ref{Lemma-v(ij)}. First,
\begin{align*}
\partial_{t} w_{ij} &=  \frac {u_{ti} u_{j} + u_{i} u_{tj}}{u^2(1-f)^{2}}
 + \frac {2 u_{t} u_{i}u_{j}f}{u^3(1-f)^{3}},\\
\partial_{k} w_{ij} &=  \frac{u_{ki} u_{j}+ u_{i} u_{jk}}{u^2(1-f)^{2}}
+ \frac{2 u_{k} u_{i} u_{j}f}{u^3(1-f)^{3}}.
\end{align*} Using
Bochner's formula, we arrive at, after differentiation,
\begin{align*}\Delta w_{ij} &=
\frac{(\Delta u)_{i} u_{j} + 2 u_{ki} u_{kj} + u_{i} (\Delta
u)_{j}} {u^2(1-f)^{2}} +  R_{ik} \frac{ u_k u_j}{u^2(1-f)^{2}} +
  R_{jk} \frac{ u_k u_i}{u^2(1-f)^{2}}\\
&\qquad+ \frac{4 (u_{ki} u_{j} +  u_{kj} u_{i})u_kf}{u^3(1-f)^{3}} +
\frac{2 u_{i} u_{j}\Delta u f}{u^3(1-f)^{3}}+ \frac{2
u_iu_ju_{k} f_k}{u^3(1-f)^{3}} + \frac{6
u_iu_ju_k^2f^2}{u^4(1-f)^{4}}.
\end{align*} Hence
$$
(\Delta -
\partial_{t}) w_{ij} = H  + R_{ik} w_{kj} + R_{jk} w_{ki},
$$ where
$$
H= \frac{2 u_{ki} u_{kj}} {u^2(1-f)^{2}} +
\frac{4 (u_{ki} u_{j} +  u_{kj} u_{i})u_kf}{u^3(1-f)^{3}}+  \frac{2
u_iu_ju_{k} f_k}{u^3(1-f)^{3}} + \frac{6
u_iu_ju_k^2f^2}{u^4(1-f)^{4}}.
$$ In the expression of $H$, we write
$4=2+2$, $6=4+2$ and hence
\begin{align*}
H&= \frac{2 (u_{ki} u_{j} + u_{kj} u_{i})u_kf}{u^3(1-f)^{3}}+\frac{4
u_iu_ju_k^2f^2}{u^4(1-f)^{4}}+\frac{2 u_iu_ju_{k} f_k}{u^3(1-f)^{3}} \\
&\qquad+\frac{2 u_{ki} u_{kj}} {u^2(1-f)^{2}} + \frac{2 (u_{ki} u_{j} +
u_{kj} u_{i})u_kf}{u^3(1-f)^{3}}+   \frac{2
u_iu_ju_k^2f^2}{u^4(1-f)^{4}}\\
&= \frac{2u_kf}{u(1-f)}\left(\frac{u_{ki} u_{j}+ u_{i}
u_{jk}}{u^2(1-f)^{2}} + \frac{2 u_{k} u_{i}
u_{j}f}{u^3(1-f)^{3}}\right)+\frac{2u_kf_k}{u(1-f)}
\cdot\frac{u_iu_j}{u^2(1-f)^{2}}\\
&\qquad+2\left(\frac{u_{ki}}{u(1-f)}+\frac{u_iu_kf}{u^2(1-f)^2}\right)
\left(\frac{u_{kj}}{u(1-f)}+\frac{u_ju_kf}{u^2(1-f)^2}\right).\end{align*}
With $u_k=uf_k$, the expression of $\partial_kw_{ij}$ and
definitions of $v_{ij}$ and $w_{ij}$, we have
$$H=\frac{2u_kf}{u(1-f)}\partial_kw_{ij}+\frac{2 |\nabla f|^{2}}{1-f} w_{ij} +
2(v_{ki} + f w_{ki}) (v_{kj} +f w_{kj}).$$ This implies the desired
result.
\end{proof}

\begin{remark}
We define the trace $w$ of $(w_{ij})$ by
$$w= tr(w_{ij}) = \frac{|\nabla u|^{2}}{u^2(1-f)^{2}}.$$ We also have
\begin{align*}
&\left(- \partial_{t}+\Delta-\frac {2f}{1-f} \nabla f\cdot\nabla\right)
v_{ij}=  (1-f) w v_{ij}\\
&\qquad\qquad+ \frac{1}{u (1-f)} [- 2
 R_{kijl} u_{kl} + R_{il} u_{jl} + R_{jl} u_{il} + ( \nabla_i R_{jl} +
\nabla_j R_{il} - \nabla_l R_{ij}) u_l ],\\
&\left(- \partial_{t}+\Delta-\frac {2f}{1-f} \nabla f\cdot\nabla\right)
w_{ij} = 2(1-f) w w_{ij}\\
&\qquad\qquad +  2(v_{ki} + f w_{ki})(v_{kj} +f w_{kj})+  R_{ik}
w_{kj} + R_{jk} w_{ki}.
\end{align*}
\end{remark}


We now prove Theorem \ref{ThmHessfixed}.

\proof[Proof of Theorem \ref{ThmHessfixed}]
{\bf Part (a).} We first perform some important calculations.
Let $p$ be a point on the manifold and $\{ x^1, ..., x^n \}$ be a local orthonormal coordinates system.
In this coordinate and using the same notations as in
Lemma \ref{Lemma-v(ij)} and Lemma \ref{Lemma-w(ij)},
the $(2, 0)$-tensor fields $v_{ij}$ and $w_{ij}$
can be regarded as $n\times n$ matrices.
Set $V = ( v_{ij})$, $ W= (w_{ij})$, $w=tr (W)$, and
$$ L= -\partial_{t} + \Delta - \frac{2f}{1-f} \nabla{f} \cdot \nabla.$$
Then, by Lemma \ref{Lemma-v(ij)} and Lemma \ref{Lemma-w(ij)},
\begin{align}
\label{eqV}
LV&= (1-f) w V + P, \\
\label{eqW}
LW&= 2(1-f)w W +2(V+fW)^{2} +Q,
\end{align}
where $P$ and $Q$ are matrices whose $(i, j)$-th components are
\begin{align}\label{eqP}\begin{split}
P_{ij}&= \frac{1}{u (1-f)} [- 2
 R_{kijl} u_{kl} + R_{il} u_{jl} + R_{jl} u_{il} + ( \nabla_i R_{jl} +
\nabla_j R_{il} - \nabla_l R_{ij}) u_l ]\\
&=- 2
 R_{kijl} v_{kl} + R_{il} v_{jl} + R_{jl} v_{il} + ( \nabla_i R_{jl} +
\nabla_j R_{il} - \nabla_l R_{ij})  \frac{u_l}{u (1-f)},
\end{split}\end{align} and
\begin{equation}
\label{eqQ}Q_{ij}= R_{ik} w_{kj} + R_{jk} w_{ki}.\end{equation}

For a constant $\alpha \in \mathbb{R}$ to be determined, we have
$$ L(\alpha V + W)= \alpha (1-f) w V + 2(1-f) w W + 2(V+fW)^{2}
+\alpha P + Q.$$
Let $\xi \in T_p M$ be a unit tangent vector at the point $p$.
We use parallel translation along
geodesics emanating from $p$ to extend $\xi$ to a smooth
vector field in the local coordinate
neighborhood. We still denote the vector field by $\xi$.  Since $V$ and $W$ are
$(2, 0)$-tensor fields,  the function
$$
\lambda = \xi^{T} (\alpha V + W) \xi \equiv (\alpha V + W) (\xi, \xi)
$$
is a well-defined smooth function in a neighborhood of $p$.
 Then
$$ L\lambda = H +  \xi^{T} (\alpha P + Q) \xi, $$
where
\be
\label{eqH}
H= \alpha (1-f) w \xi^{T} V \xi +2 (1-f) w \xi^{T} W \xi + 2 |(V+fW) \xi|^{2}.
\ee
By $ \alpha \xi^{T} V \xi = \lambda - \xi^{T} W \xi$,
we have
\begin{align*}
H&=  (1-f) w (\lambda - \xi^{T} W \xi) + 2(1-f) w \xi^{T} W \xi + 2 |(V +f W) \xi|^{2} \\
&=  (1-f) w \lambda + (1-f) w \xi^{T} W \xi + 2 |(V+ fW)\xi|^{2}.
\end{align*}
To simplify the last term further, we fix the point $p$ and assume $\xi$ is the vector field generated,
via parallel translation through geodesics emanating from $p$,
by an eigenvector of $\alpha V + W$ at $p$, i.e., at $p$,
$$ (\alpha V + W) \xi =\lambda \xi.$$
Then
$$ (V + fW) \xi = \frac{\lambda}{\alpha} \xi -\frac{1}{\alpha} W \xi + f W \xi= \frac{\lambda}{\alpha} \xi -\left(\frac{1}{\alpha}- f\right) W \xi, $$
and hence
$$ |(V+ fW) \xi|^{2} = \frac{\lambda^{2}}{\alpha^{2}} - \frac{2 \lambda}{\alpha} \left( \frac{1}{\alpha} - f\right) \xi^{T} W \xi + \left( \frac{1}{\alpha} -f\right)^{2} |W \xi|^{2}.$$
Hence
\begin{align*}H&= \frac{2 \lambda^{2}}{\alpha^{2}} + \lambda \left( w- \frac{4}{\alpha^{2}} \xi^{T} W \xi \right)
- f \lambda \left( w- \frac{4}{\alpha} \xi^{T} W \xi\right)\\
&\qquad+ (1-f) w \xi^{T} W \xi + 2 \left( \frac{1}{\alpha}- f\right)^{2} |W \xi|^{2}.\end{align*}
The last two terms are independent of $\lambda$ and nonnegative.
Hence
\be
\label{H>}
 H \geq \frac{2 \lambda^{2}}{\alpha^{2}} + \lambda \left( w- \frac{4}{\alpha^{2}} \xi^{T} W \xi\right) - f \lambda \left( w- \frac{4}{\alpha} \xi ^{T} W \xi\right).
\ee For the last term, we note $W$ is a rank-one matrix and hence
$$ \xi^{T} W \xi \leq w.$$
With $f <0$ and choosing $\alpha  \ge 4$,  the third term in nonnegative, if $\lambda \geq 0$.
If $\lambda \geq 0$, the second term is also nonnegative. Hence
\be
\label{H>2lamb2} H \geq \frac{2 \lambda^{2}}{\alpha^{2}}.
\ee
In summary, if $ \lambda \geq 0$, then at the point $p$,
\begin{equation}
\label{Llamb}
 L \lambda \geq \frac{2 \lambda^{2}}{\alpha^{2}} +  \xi^{T} (\alpha P + Q) \xi.
 \end{equation}

Now we proceed to prove Part (a).
Let $\tau$ be a universal constant to be fixed later. With
$\alpha=4$, suppose the 2 form
$$\alpha V + W-\frac{\tau}{t} g$$ assumes its largest nonnegative eigenvalue
at  a space-time point $ (p_1, t_1)$, with $t_1 >0$.
Let $\xi$ be a unit eigenvector at $p_1$. We use parallel translation along
geodesics emanating from $p_1$ to extend $\xi$ to a smooth vector field which is still denoted by $\xi$. Set, in a local coordinate
\be
\label{defmu} \mu = \xi^{T} (\alpha V + W-\frac{\tau}{t} g) \xi.
\ee
and
\be
\label{deflamb} \lambda = \xi^{T} (\alpha V + W) \xi.
\ee
Then, both $\mu$ and $\lambda$ are smooth functions in a space time neighborhood of
$(x_1, t_1)$.  Also
$$L \mu = L \left(\lambda-\frac{\tau}{t}\right) = L \lambda- \frac{\tau}{t^2} = H -\frac{\tau}{t^2} + \xi^{T} (\alpha P + Q) \xi.
$$ Here $H$ is given by (\ref{eqH}).
We now evaluate at $(p_1, t_1)$.
Since $\lambda -\tau/t$ has its nonnegative maximum at $(p_1, t_1)$, we have,  by (\ref{Llamb}),
$$ 0 \geq  L\left(\lambda-\frac{\tau}{t}\right)
\geq \frac{2\lambda^{2}}{\alpha^{2}}-\frac{\tau}{t^2}-|\xi^{T}
(\alpha P + Q) \xi | \quad\text{at }(p_1,t_1), $$ or \be
\label{2lamb2/a2<}
\frac{2\lambda^{2}}{\alpha^{2}}\le\frac{\tau}{t^2}+|\xi^{T}
(\alpha P + Q) \xi | \quad\text{at }(p_1,t_1). \ee In order to
bound $\mu$ and $\lambda$ from above, we need to find an upper bound for $
|\xi^{T} (\alpha P + Q) \xi | $ at $(p_1, t_1)$.

Let $\xi=(\xi_1,
..., \xi_n)$.  By (\ref{eqP}) and (\ref{eqQ}), we obtain
\begin{align*}
&| \xi^{T}   (\alpha P + Q) \xi |\le | \alpha\xi^{T} P \xi |
+| \xi^{T} Q \xi |\\
&\quad\le \left| \xi_i \xi_j \alpha \left[- 2
R_{kijl} v_{kl} + R_{il} v_{jl} + R_{jl} v_{il} + ( \nabla_i R_{jl} +
\nabla_j R_{il} - \nabla_l R_{ij})  \frac{u_l}{u (1-f)} \right]\right|\\
&\quad\qquad+ \left|\xi_i \xi_j (R_{ik} w_{kj} + R_{jk} w_{ki})  \right| \\
&\quad\le  \left| \xi_i \xi_j \alpha \left[- 2
 R_{kijl} v_{kl} + R_{il} v_{jl} + R_{jl} v_{il}  \right] \right|  + C |\nabla Ric | \sqrt{|W|} + C |Ric| |W|.
\end{align*}
Writing $\alpha v_{kl} = \alpha v_{kl} + w_{kl} - w_{kl}$ etc  in the last line, we deduce
\be
\label{|aP+Q|}
\al
&| \xi^{T}   (\alpha P + Q) \xi |\\
&\quad\le  \left| \xi_i \xi_j  \left[- 2
R_{kijl} (\alpha v_{kl} + w_{kl}) + R_{il} (\alpha v_{jl}  + w_{jl})+
R_{jl} (\alpha v_{il}+ w_{il}) \right] \right|  \\
&\qquad\quad+ \left| \xi_i \xi_j  \left[- 2
R_{kijl}  w_{kl} + R_{il}  w_{jl}+ R_{jl}  w_{il} \right] \right|
+ C |\nabla Ric | \sqrt{|W|} + C |Ric| |W|.
\eal
\ee
At the point $p_1$, we can choose a coordinate system so that the matrix $\alpha V + W$ is diagonal.
Let $\mu_1, \cdots, \mu_n$ be the eigenvalues
of the matrix $\alpha V+W-\frac{\tau}{t} I$, listed in the increasing order.
Without loss of generality, we assume $\mu_1<0$ and $\mu_n>0$.
Then
$$
\al
&| R_{kijl} (\alpha v_{kl} + w_{kl}) | \le | R_{kijl} (\alpha v_{kl} + w_{kl} - \frac{\tau}{t}
\delta_{kl}) | + | R_{kijl}
\delta_{kl}|  \frac{\tau}{t}\\
&\le | \sum^{n}_{k=1}  R_{kijk} (\alpha v_{kk} + w_{kk}- \frac{\tau}{t})|
+ C |Rm|   \frac{\tau}{t} \le
C |Rm| (\mu_n + |\mu_1|) + C |Rm|   \frac{\tau}{t}.
\eal
$$
Similarly
$$
|R_{il} (\alpha_{il} v_{jl}  + w_{jl})| \le C |Ric| (\mu_n+ |\mu_1|) + C |Ric|   \frac{\tau}{t}.
$$
Combining the last few inequalities, we deduce
$$
| \xi^{T}   (\alpha P + Q) \xi | \le C |Rm| (\mu_n +|\mu_1|)+ C   |\nabla Ric | \sqrt{|W|} + C  |Rm| |W| + C |Rm|   \frac{\tau}{t}.
$$
In the following, we set $K_1=|Rm|_{L^\infty}$ and $K_2=|\nabla Ric|_{L^\infty}$. Then
\begin{align*}| \xi^{T}   (\alpha P + Q) \xi | &\le C K_1 (\mu_n +|\mu_1|
+ \frac{\tau}{t})+ C K_2 \sqrt{|W|} + C K_1 |W|\\
&\le C K_1 (\mu_n +|\mu_1| +  \frac{\tau}{t} )+ C K_2  + C (K_1+K_2) |W|.
\end{align*}
Observe that
\begin{align*} \mu_1 + (n-1) \mu_n &\ge \mu_1+\cdots+\mu_n \\
&= \operatorname{tr} \left(\frac{\alpha
u_{ij}}{u(1-f)}+\frac{u_iu_j}{u^2(1-f)^2} - \frac{\tau}{t} \delta_{ij} \right) \ge \frac{\alpha
\Delta u}{u(1-f)}  - n \frac{\tau}{t}.
\end{align*}
 Hence,
$$|\mu_1| \le (n-1) \mu_n -\frac{\alpha \Delta u}{u(1-f)} + n \frac{\tau}{t}.$$
Therefore
$$| \xi^{T}   (\alpha P + Q) \xi |\le C K_1 \left(\mu_n -\frac{\alpha \Delta u}{u(1-f)}
+ \frac{\tau}{t} \right)  + C K_2  + C (K_1+K_2) |W|.
$$

By \cite{Li-Yau1986}, we have
$$\frac{|\nabla u|^2}{u^2} - 2 \frac{ u_t }{u} \le \frac {c_1}t+K_0,$$
where $Ric\ge -K_0$ for some $K_0\ge 0$.
With $u_t =\Delta u$, we get
$$-\frac{\Delta u }{u} \le \frac {c_1}t+K_0.$$
Since $0\le u<A$, we have
$$\frac{1}{1-f}\le  1.$$
Therefore, at $(p_1, t_1)$, it holds
$$
| \xi^{T}   (\alpha P + Q) \xi | \le C K_1\left (\mu_n +\frac{1+\tau}{t}\right)+C(K_1K_0+K_2)
+ C (K_1+K_2) |W|.
$$
By the definition of $W=(w_{ij})$, we have
$$|W|\le \frac{|\nabla u|^2}{u^2(1-f)^2}.$$
By Theorem 1.1 in \cite{Hamilton1993}, we obtain
$$|W|\le C\left(\frac1t+K_0\right),$$
and hence
$$
| \xi^{T}   (\alpha P + Q) \xi | \le C (K_1+K_2)\left (\mu_n +\frac{1+\tau}{t}\right)+C(K_1K_0+K_2+K_2K_0).
$$
Since $\mu=\mu_n<\lambda$ at $(p_1, t_1)$, this shows, by (\ref{2lamb2/a2<}), that
$$  \frac{2\lambda^{2}}{\alpha^{2}}\le\frac{\tau}{t^2}
+ C (K_1+K_2)\left (\lambda +\frac{1+\tau}{t}\right)+C(K_1K_0+K_2+K_2K_0)
\quad\text{at }(p_1,t_1). $$
A simple application of the Cauchy inequality yields
$$\frac{\lambda}{\alpha}\le \frac{\sqrt{\tau+1}}t+ B\quad\text{at }(p_1,t_1), $$
where $B$ is a nonnegative constant depending only on $K_0, K_1$ and $K_2$
with the property that $B=0$ if $K_0=K_1=K_2=0$.
Then
$$\lambda-\frac\tau t\leq (\alpha\sqrt{\tau+1}-\tau)\frac1t+\alpha B\le \alpha B\quad\text{at }(p_1,t_1),$$
by choosing $\tau$ sufficiently large. In fact, for $\alpha=4$, we can
take $\tau=8+4\sqrt5$.

By definition,  $\mu=\lambda-\frac\tau
t$ at $(p_1, t_1)$ is the largest eigenvalue of the $2$ form
$\alpha V + W - \frac{\tau}{t} g$ on $M \times (0, T]$.
Therefore, given any  unit tangent vector $\eta \in T_x M$, $x
\in M$, it holds
\[
\eta^{T} (\alpha V + W) \eta - \frac{\tau}{t} g(\eta, \eta) \le (\lambda-\frac\tau
t) |_{(p_1, t_1)} \le
\alpha B\quad \text{in} \ M\times (0, T).
\] Thus
$$t  \eta^{T} V \eta  \leq \frac\tau\alpha+Bt.$$
This proves part (a) of the theorem.

\medskip

{\bf Part (b).}
Now we localize the result in part (a). It is unexpected  that the local bound is
different from the global one.  We point it out in the remark below that the local
form is also sharp in general.

Let $\psi$ be a cutoff function which
will be specified later. Then, for any smooth function $\eta$, we
have
\begin{align*}
\partial_{t} (\psi \eta) &=  \partial_{t} \psi \ \eta + \psi \partial_{t} \eta, \\
\nabla (\psi \eta) &=  \nabla \psi \ \eta + \psi  \nabla \eta, \\
\Delta (\psi \eta) &=  \Delta \psi \ \eta + 2 \nabla \psi \nabla \eta + \psi  \Delta \eta.
\end{align*}
Hence
\begin{align*}
\psi L \eta &=  - \psi \ \partial_{t} \eta + \psi  \Delta \eta - \psi \frac{2 f}{ 1-f} \nabla f\cdot \nabla \eta \\
&=  - (\partial_{t} (\psi \eta) - \eta  \partial_{t} \psi) + \Delta ( \psi \eta) - \eta\Delta \psi
- 2 \nabla \psi\cdot \nabla \eta - \frac{2 f}{1-f} \nabla f (\nabla (\psi \eta) - \eta\nabla \psi) \\
&=  - \partial_{t} (\psi \eta) + \Delta (\psi \eta) - \frac{2 f}{1-f}\nabla f \cdot\nabla (\psi \eta)
+ \eta  \partial_{t} \psi - \eta \Delta \psi \\
&\quad+\frac{2 f}{1-f}\eta \nabla f\cdot \nabla \psi  -2 \nabla \psi \cdot\nabla \eta.
\end{align*}
For the last term, we write
\begin{align*}
\nabla \psi\cdot\nabla \eta &=  \frac{\nabla \psi}{\psi} \psi  \nabla \eta
= \frac{\nabla \psi}{\psi} (\nabla (\psi \eta) - \nabla \psi \ \eta) \\
&=  \frac{\nabla \psi}{\psi} \nabla (\psi \eta) - \frac{|\nabla \psi|^{2}}{\psi} \eta.
\end{align*}
Hence
\begin{align*}
\psi L \eta &=  - \partial_{t} (\psi \eta) + \Delta (\psi \eta) - \frac{2 f}{1-f} \nabla f\cdot \nabla (\psi \eta)
- \frac{2 \nabla \psi}{\psi} \nabla (\psi \eta) \\
& \qquad +\eta  \partial_{t} \psi - \eta  \Delta \psi + \eta \frac{2 f}{1-f} \nabla f \cdot \nabla \psi
+ \frac{2 |\nabla \psi|^{2}}{\psi} \eta.
\end{align*}
Set
\be
\label{defL1} L_1 = - \partial_{t} + \Delta - \frac{2 f}{1-f} \nabla f \cdot \nabla - \frac{2 \nabla \psi}{\psi} \ \nabla.
\ee
Then
$$ \psi  L \eta = L_1 ( \eta \psi) - \eta L_1\psi, $$
or
$$ L_1(\eta \psi) = \psi L \eta + \eta L_1\psi. $$
With $\lambda$ introduced before in (\ref{deflamb}), we have
\be
\label{L1psilamb} L_1(\psi \lambda) = \psi L \lambda + \lambda L_1\psi = \psi[ H +\xi^T(\alpha P+Q)\xi]
+ \lambda L_1\psi.
\ee Here $H$ is given by (\ref{eqH}).
Now we analyze $L_1\psi$. We write
$$ L_1\psi = -\partial_{t} \psi + \Delta \psi - \frac{2 |\nabla \psi|^{2}}{\psi} - \frac{2 f}{1-f} \nabla f\cdot \nabla \psi.$$
The first three terms are obviously bounded by choosing suitable $\psi$. For the last one, we write
$$ - \frac{2 f}{1-f} \nabla f \cdot \nabla \psi = - \frac{2 f}{1-f} \sqrt{\psi} \ \nabla f \cdot \frac{\nabla \psi} {\sqrt{\psi}}.$$
Note $\displaystyle{\frac{-f}{1-f} < 1}$ and $\displaystyle{\frac{\nabla \psi}{\sqrt{\psi}}}$ is bounded. We need to control
$$ \sqrt{\psi} \nabla f.$$
To this end, we recall the equation for $f$
$$ - \partial_{t} f+ \Delta f= - |\nabla f|^{2}.$$
Then
$$
L f = -\partial_{t} f + \Delta f - \frac{2 f}{1-f} |\nabla f|^{2} = -|\nabla f|^{2} - \frac{2 f}{1-f} |\nabla f|^{2}
=  \frac{-1-f}{1-f} |\nabla f|^{2}.$$
Note
$$\frac{-1-f}{1-f} \geq \frac{1}{2} \quad \text{if}  \ f \leq -3.$$
Then
$$ L f \geq \frac{1}{2} |\nabla f|^{2}.$$
Hence, we obtain
$$ L_1f = L f - \frac{2 \nabla \psi}{\psi} \nabla f = L f - \frac{2}{\psi} \frac{\nabla \psi}{\sqrt{\psi}} \ \sqrt{\psi} \nabla f ,$$
and then
\be
\al
\label{L1f>}
\psi L_1f &=  \psi L f - 2 \frac{\nabla \psi}{\sqrt{\psi}} \ \sqrt{\psi} \nabla f
\geq  \frac{1}{2} \psi |\nabla f|^{2} - 2 \frac{\nabla \psi}{\sqrt{\psi}} \ \sqrt{\psi} \nabla f \\
&\geq  \frac{1}{4} \psi |\nabla f|^{2} - C \frac{|\nabla \psi|^{2}}{\psi}.
\eal
\ee

Now we consider, for some constant $\beta \in \mathbb{R}^+$ to be determined,
$$ \psi L_1(\psi \lambda + \beta f) = \psi^{2} H +
\psi^2\xi^T(\alpha P+Q)\xi+ \psi \lambda L_1\psi + \beta \psi L_1f.$$
In the following, we consider eigenvalues of the 2 form
$$ \psi (\alpha V + W) + \beta f g.$$
If $\xi$ is an eigenvector of $\psi (\alpha V + W)+ \beta f g$ at
some point $(x, t)$, then
$$ [\psi (\alpha V + W) + \beta f g] \xi = \mu \xi,$$
or in local coordinates,
$$ \psi (\alpha V + W) \xi= (\mu - \beta f ) \xi.$$
If $\psi (x, t) \neq 0$, then $\xi$ is also an eigenvector of $\alpha V + W$. Hence
$$ \mu = \psi \lambda + \beta f. $$
Again we extend $\xi$ to a vector field around $x$ by parallel
transporting along geodesics starting from $x$. The vector field
is still denoted by $\xi$.

Let $\Omega$ be the parabolic cube given by
$$ \Omega= Q_{R,T}(x_0,t_0) = B(x_{0}, R) \times (t_{0} - T, t_{0}].$$
Let
$$
\mu = \xi^{T} [ \psi (\alpha V + W) + \beta f g] \xi = \psi
\lambda + \beta f.
$$ Then
$$ \mu \big|_{\partial_{p} \Omega} = \beta f \big|_{\partial_{p} \Omega} < 0.  $$
We will estimate $\mu$ from above. Recall from (\ref{L1psilamb}) and  $\mu
=\lambda + \beta f$ that
\be
\label{psiL1mu=2nd}
 \psi L_1 \mu = \psi^{2} H +\psi^2\xi^T(\alpha P+Q)\xi+
\psi\lambda L_1 \psi  + \beta \psi L_1f.
\ee
We first have, from (\ref{H>2lamb2}), that
$$ \psi^{2} H \geq \frac{2}{\alpha^{2}} (\psi \lambda)^{2} + \psi^{2} \lambda
\left(w - \frac{4}{\alpha^{2}} \xi^{T} W \xi\right)
- f \psi^{2} \lambda \left(w- \frac{4}{\alpha}\xi ^{T} W \xi\right).$$
At points where $\mu \geq 0$, we have
$$ \psi \lambda + \beta f \geq 0,$$
and hence $ \psi \lambda \geq 0$. Then
$$ \psi ^{2} H \geq \frac{2}{\alpha^{2}} (\psi \lambda)^{2}.$$
By this, (\ref{L1f>}) and (\ref{psiL1mu=2nd}), we deduce,
\begin{align*}
\psi L_1\mu &\geq  \frac{2}{\alpha^{2}} (\psi \lambda)^{2} +
\psi^2\xi^T(\alpha P+Q)\xi+ \beta \left[ \frac{1}{4} \psi |\nabla f|^{2}
- C \frac{|\nabla \psi|^{2}}{\psi}\right] \\
&\qquad - \left[ |\partial_{t} \psi| + |\Delta \psi|
+ \frac{2 |\nabla \psi|^{2}}{\psi}
+2 \sqrt{\psi} |\nabla f| \cdot\frac{\nabla \psi}{\sqrt{\psi}}\right]
\psi \lambda.
\end{align*}
For the last term, we use the Cauchy inequality to control the
$\psi \lambda$ factor by the first term to get
\begin{align*}
\psi L_1 \mu \geq & \frac{1}{\alpha^{2}} (\psi \lambda)^{2} + \psi^2\xi^T(\alpha P+Q)\xi
+\beta \left( \frac{1}{4} \psi |\nabla f|^{2} - C \frac{|\nabla \psi|^{2}}{\psi}\right) \\
& - C \left( |\partial_{t} \psi| + |\Delta \psi| + \frac{2 |\nabla \psi|^{2}}{\psi}\right)^{2}
- C \psi |\nabla f|^2
\frac{|\nabla \psi|^2}{\psi}.
\end{align*}
We now take
\be
\label{beta=1}
\beta=c\sup  \frac{|\nabla \psi|^{2}}{\psi}.
\ee
If $c$ is sufficiently large, we have
$$\beta \left( \frac{1}{4} \psi |\nabla f|^{2} - C \frac{|\nabla \psi|^{2}}{\psi}\right)- C \psi |\nabla f|^2
\frac{|\nabla \psi|^2}{\psi}\ge -
C'\sup  \frac{|\nabla \psi|^{4}}{\psi^2},$$
and hence
\begin{align*}
\psi L_1 \mu &\geq  \frac{1}{\alpha^2} (\psi \lambda)^{2}
+\psi^2\xi^T(\alpha P+Q)\xi+ \frac{1}{8} \beta \psi |\nabla f|^{2} \\
&\quad\quad - C \left[ |\partial_{t} \psi|+ |\Delta \psi|
+\frac{|\nabla \psi|^{2}}{\psi} \right]^{2}-
C\sup  \frac{|\nabla \psi|^{4}}{\psi^2}.
\end{align*}

Let $\psi$ be a cut off function supported in the
space-time cube $Q_{R, T}(x_0, t_0)$ such
that $\psi=1$ in the cube of half the size $Q_{R/2, T/2}(x_0, t_0)$.
We also require that
\[
 |\nabla \psi| \le \frac{C}{R}, \quad |\Delta
\psi | \le C \frac{K_0+1}{R^2}, \quad \frac{|\partial_t \psi
|}{\sqrt{\psi}} \le C \frac{1}{T}, \quad \frac{|\nabla
\psi|^2}{\psi} \le \frac{C}{R^2}.
\] Here $K_0$ is a bound on $|Ric|$ as before. Also the quotients are regarded
as $0$ when $\psi=0$ somewhere.  Without loss of generality we
just take $t_0=T$. We also require that $\psi$ is supported in the
slightly shorter space time cube $Q_{R, 3T/4}(x_0, t_0)$.  As
usual, the cutoff function can be constructed from the distance
function, which is not always smooth. One can either mollify the
distance by convoluting with a smooth kernel or use the well known
trick by Calabi to get around singular points of the distance.

Let $\mu_{0}$ be the maximal eigenvalues of $\psi (\alpha V+ W) + \beta f g$
with a unit eigenvector $\xi$.
Assume $\mu_{0}$ is taken at the space-time point
$(p_1, t_1)$. Again, by parallel translation,  we extend
$\xi$ to a vector field in a neighborhood of $p_1$, which is still denoted by $\xi$.
We are interested only in the case $\mu_{0} > 0$. Since $f \le 0$ and $\psi=0$ on
the parabolic boundary of $\Omega$, we know that $p_1$ must lie in the interior of $B(p, R)$. Define a function $\mu=\mu(x, t)$ around $(p_1, t_1)$ by
\[
\mu=
\xi^T \left( \psi (\alpha V+ W) + \beta f g \right) \xi.
\]
Since $(p_1, t_1)$ is a maximum point of $\mu$, we have
\begin{align*}
0 \geq \psi L_1 \mu &\geq \frac{1}{\alpha^{2}} (\psi \lambda)^{2} +\psi^2\xi^T(\alpha P+Q)\xi\\
&\qquad- C \left[ |\partial_{t} \psi|
+ |\Delta \psi| + \frac{|\nabla \psi|^{2}}{\psi} \right]^{2}-
C\sup  \frac{|\nabla \psi|^{4}}{\psi^2}.\end{align*}
Hence, at $(p_1, t_1)$, it holds
\be
\label{psilamb2<case1} \frac{1}{\alpha^{2}} (\psi \lambda)^{2} \le\psi^2|\xi^T(\alpha P+Q)\xi|+C \left(\frac{1}{T} + \frac{1}{R^{2}} \right)^2.
\ee  From (\ref{|aP+Q|}),
\[
\al
&\psi | \xi^{T}   (\alpha P + Q) \xi |\\
&\quad\le   \left| \xi_i \xi_j  \left[- 2
R_{kijl} \psi (\alpha v_{kl} + w_{kl}) + R_{il} \psi (\alpha v_{jl}  + w_{jl})+ R_{jl} \psi (\alpha v_{il}+ w_{il}) \right] \right|  \\
&\qquad\quad+ \psi \left| \xi_i \xi_j  \left[- 2 R_{kijl}  w_{kl} +
R_{il}  w_{jl}+ R_{jl}  w_{il} \right] \right|  + C\psi |Ric| |W| +  C\psi |\nabla Ric| \sqrt{|W|}\\
&\quad\le \left| \xi_i \xi_j  \left[- 2
R_{kijl} \psi(\alpha v_{kl} + w_{kl}) + R_{il} \psi(\alpha v_{jl}  + w_{jl})+ R_{jl} \psi (\alpha v_{il}+ w_{il}) \right] \right| \\
&\qquad \quad + C (K_0+K_1+K_2) \psi |W| + C K_2.
\eal
\]
By splitting off a term $\beta f$, we obtain
\[
\al
&\psi | \xi^{T}   (\alpha P + Q) \xi |\\
&\le \bigg| \xi_i \xi_j  \big[- 2
R_{kijl} \{ \psi (\alpha v_{kl} + w_{kl})- \beta f \delta_{kl}\} + R_{il} \{\psi (\alpha v_{jl}  + w_{jl})- \beta f \delta_{kl}\}\\
&\qquad \qquad + R_{jl} \{\psi (\alpha v_{il}+ w_{il}) - \beta f \delta_{kl}\} \big] \bigg| \\
& \qquad +  C (K_0+K_1) \beta \psi |f|+ C (K_0+K_1+K_2) \psi |W| + C\psi K_2.
\eal
\] Let $\mu_1, ..., \mu_n$ be the eigenvalues of the $2$ form
$\psi(\alpha V + W) + \beta f g$
at $(p_1, t_1)$, which are listed in increasing order. We assume without loss
of generality that $\mu_1<0$. Then the above inequality implies
\begin{align*}
\psi | \xi^{T}   (\alpha P + Q) \xi | &\le C (K_0 + K_1) (\mu_n + |\mu_1|) +
 C (K_0+K_1) \beta  |f|\\
 &\qquad + C (K_0+K_1+K_2) \psi |W| + C \psi K_2.
\end{align*}
Observe that
\begin{align*} \mu_1 + (n-1) \mu_n &\ge \mu_1+\cdots+\mu_n \\
&= \operatorname{tr}\left[  \psi\left(\frac{\alpha
u_{ij}}{u(1-f)}+\frac{u_iu_j}{u^2(1-f)^2} \right) + \beta f \delta_{ij} \right] \\
&\ge \psi \frac{\alpha
\Delta u}{u(1-f)}  + n \beta f .
\end{align*}
 Hence,
$$|\mu_1| \le (n-1) \mu_n -\psi \frac{\alpha \Delta u}{u(1-f)} + n \beta |f|.$$
Therefore,
\begin{align}
\label{(aP+Q)case1<}\begin{split}
\psi | \xi^{T}   (\alpha P + Q)
\xi |&\le C K_1 \left(\mu_n - \psi \frac{\alpha \Delta u}{u(1-f)}
 \right) \\
 &\qquad + C (K_1+K_2) \psi |W| + C K_2 + C K_1 \beta |f|.
\end{split}\end{align}

Since $Ric \ge -K_0$, by \cite{Li-Yau1986} and our choice of the cutoff function
$\psi$, we have, for any $a>1$,
$$
\psi^2\left(\frac{|\nabla u|^2}{u^2} - a\frac{ u_t }{u} \right)\le C\left(\frac1T+\frac{1}{R^2}+K_0\right).
$$ We note that the time $1/T$ is actually $1/t$ in \cite{Li-Yau1986}. However,
since our cutoff function is supported in a shorter cube,
these two terms are equivalent.
With $u_t =\Delta u$ and $\mu=\psi \lambda + \beta f \le \psi \lambda$, we get, at $(p_1, t_1)$,
\begin{align*}
\psi^2| \xi^{T}   (\alpha P + Q) \xi | &\le C K_1\psi^2\lambda +CK_1\left(\frac1T+\frac{1}{R^2}\right)\\
&\qquad+C(K_1K_0+K_2)
+ C K_2 \psi^2|W|+ C K_1 \beta |f|.
\end{align*}
From the definition of $W=(w_{ij})$, we have
$$|W|\le \frac{|\nabla u|^2}{u^2(1-f)^2}.$$
By Theorem 1.1 in \cite{Souplet-Zhang2006}, we obtain
$$\psi^2|W|\le C\left(\frac1T+\frac{1}{R^2}+K_0\right),$$
and hence
\begin{align*}
\psi^2| \xi^{T}   (\alpha P + Q) \xi | &\le C K_1\psi^2\lambda +C(K_1+K_2)\left(\frac1T+\frac{1}{R^2}\right)\\
&\qquad +C(K_1K_0+K_2+K_2K_0) + C K_1 \beta |f|.
\end{align*}
Substituting this to (\ref{psilamb2<case1}), we find, at $(p_1, t_1)$, that
\begin{align*}
\frac{1}{\alpha^{2}} (\psi \lambda)^{2}&\le C K_1\psi^2\lambda
+C(K_1+K_2)\left(\frac1T+\frac{1}{R^2}\right)
+C\left(\frac1T+\frac{1}{R^2}\right)^2\\
&\qquad+C(K_1K_0+K_2+K_2K_0)+ C K_1 \beta |f|,\end{align*}
and hence
$$\psi \lambda\le C \left(\frac1T+\frac{1}{R^2}+B\right) + C \sqrt{K_1 \beta |f|},$$
where $B$ is a nonnegative constant depending only on $K_0, K_1$ and $K_2$
with the property that $B=0$ if $K_0=K_1=K_2=0$. This implies
\[
\mu = \psi \lambda + \beta f \le C \left(\frac1T+\frac{1}{R^2}+B\right) + C \sqrt{K_1 \beta |f|}
+ \beta f.
\]
Since $f<0$, we  know that $C \sqrt{K_1 \beta |f|}
+ \beta f \le C K_1$. Thus
$$ \mu_{0} = \mu \big|_{(p_1, t_1)} = (\psi\lambda + \beta f) \big|_{(p_1, t_1)} \leq C \left( \frac{1}{T} + \frac{1}{R^{2}}+B \right).$$
Here $B$ may have changed from the last line.
Therefore
$$ \mu \leq C \left( \frac{1}{T} + \frac{1}{R^{2}} +B\right) \quad \text{in} \ Q_{R,T}.$$
Hence, for any unit tangent vector $\xi$ at $x$ with $(x, t) \in Q_{R,T}$, it holds
$$ \psi \xi^{T} (\alpha V + W) \xi + \beta f \leq C \left( \frac{1}{T} + \frac{1}{R^{2}} +B\right) \quad
\text{in} \ Q_{R,T},$$
or
$$ \psi \xi^{T} (\alpha V + W) \xi \leq C \left(\frac{1}{T} + \frac{1}{R^{2}} +B\right)+ \beta |f|, \quad \text{in} \ Q_{R,T}. $$
Recall from (\ref{beta=1}) that $\beta= \frac{C}{R^2}$, we then have
$$ \psi \xi^{T} V \xi \leq C \left( \frac{1}{T} + \frac{1}{R^{2}} +B\right) (1-f),$$
and hence
$$ \psi \frac{u_{ij} \xi_{i} \xi_{j}}{u} \leq C \left( \frac{1}{T} + \frac{1}{R^{2}} +B\right) (1-f)^{2}.$$
This implies the desired estimate. \qed

\begin{remark} When we compare the local version with the global version,
we note an extra power of $\displaystyle{1+ \log\frac{A}{u}}$ in
Part (b) of Theorem \ref{ThmHessfixed}. This turns out to be optimal. Consider
$x_{0} =2 $, $R=1$, $t_{0}=2$, $T=1$ and $Q_{1,1} (2,2) = [1,3]
\times [1,2] \subset \mathbb{R}_{x} \times \mathbb{R}_{t}.$ For $a>0$, set
$$ u(x,t) = e^{ax+a^2t}.$$
This is a positive solution of the heat equation in $Q_{1,1} (2,2)$.
Note
$$ \frac{u_{xx}(2,2)}{u(2,2)} = a^{2},$$
and
$$ A=\sup_{Q_{1,1} (2,2)} u = e ^{3a + 2a^{2}}, \quad
\log \frac{A}{u(2,2)} = \log \frac{e^{3a+2a^{2}}}{e^{2a+2a^{2}}}= a. $$
Hence, $\displaystyle{\frac{u_{xx}(2,2)}{u(2,2)}}$ and
$\displaystyle{\left[ \log \frac{A}{u(2,2)}\right]^{2}}$ have the same order in $a$.
\end{remark}

\section{Heat Equations under Ricci Flow}\label{sec-RicciFlow}

In this section we consider the heat equation coupled with the Ricci
flow on a manifold $M$, over a time interval $(0, T]$,
\begin{equation}
\label{heateqRicci}
\begin{cases}
\Delta u - \partial_t u = 0\\
\partial_t g = - 2 Ric.
\end{cases}
\end{equation}
The heat equation and its conjugate have served as a fundamental
tool in the theory of Ricci flow developed by Hamilton and
Perelman. The two authors and others have derived gradient
estimate for positive solutions of the heat and conjugate heat
equation. See the paper \cite{CH:1}, \cite{P:1}, \cite{Ni:1},
\cite{CaH:1}, \cite{Ca:1} and \cite{BCP:1} for instance.
 Here we prove an upper bound on the
Hessian of the log solution, which seems to be missing as the
fixed metric case. The general idea of the proof is similar to
that in the previous section. However, since the metric is
evolving, there will be extra terms to deal with, especially the
term $R_{ij} u_{ij}$. To treat this term, we need to use the
latest result in \cite{BCP:1} which provides a Li-Yau type
gradient bound in the Ricci flow case.

We will keep the same notations as in the last section. First let
us derive the evolution equation for $(v_{ij})$ and $(w_{ij})$. In
this situation, the corresponding equations for $(v_{ij})$ have
fewer terms than those in Lemmas \ref{Lemma-v(ij)} and
\ref{Lemma-w(ij)} in the fixed metric case. More specifically the
terms involving the gradient of the Ricci curvature drops out.
This is due to a cancellation introduced by the Ricci flow. The
equation for $(w_{ij})$ will formally stay the same.

\begin{lemma}
\label{Lemma-v(ij)RF} Suppose $u$ is a positive solution to
(\ref{heateqRicci}) such that $0<u \le A$.  Set $f=\log(u/A)$ and
$$ v_{ij} = \frac{u_{ij}}{ u(1-f)}$$  with the matrix $V=(v_{ij})$ representing the $2$ form $\frac{Hess \, u}{ u(1-f)}$ in local coordinates. Then,
$$
\begin{aligned}
 &(- \partial_{t}+\Delta-\frac {2f}{1-f} \nabla
f\cdot\nabla) v_{ij} \\
&= \frac{|\nabla f|^{2}}{1-f} v_{ij} + \frac{1}{u (1-f)} [- 2
 R_{kijl} u_{kl} + R_{il} u_{jl} + R_{jl} u_{il} ].
\end{aligned}
$$
\end{lemma}
\proof

This is similar to  that of Lemma \ref{Lemma-v(ij)}. The only difference is that the
commutation formula now is
\[
 - \partial_t u_{ij}\Delta u_{ij}
 +  2 R_{kijl} u_{kl} - R_{ik} u_{jk} - R_{jk} u_{ik} = 0.
\]See p109 of \cite{CLN:1} e.g.
\qed

\begin{lemma}
\label{Lemma-w(ij)RF} Suppose $u$ is a positive solution to
(\ref{heateqRicci}) such that $0<u \le A$.  Set $f=\log(u/A)$ and
$$
 w_{ij} = \frac{u_{i} u_{j}}{u^2(1-f)^{2}}
 $$  with the matrix $W=(w_{ij})$ representing the $2$ form $\frac{du \otimes du}{ u^2(1-f)^2}$ in local coordinates. Then,
\[
\al
(- \partial_{t}+&\Delta-\frac {2f}{1-f} \nabla f\cdot\nabla) w_{ij}\\
 &=  \frac{2 |\nabla f|^{2}}{1-f} w_{ij}
+  2(v_{ki} + f w_{ki})(v_{kj} +f w_{kj}) + R_{ik} w_{kj} + R_{jk} w_{ki}.
\eal
\]
\end{lemma}

\proof

This is formally identical to  that of Lemma \ref{Lemma-w(ij)}.
The reason is that $W=(w_{ij})$ represents the $2$ form
$\frac{du \otimes du}{ u^2(1-f)^2}$ and
$d u$
commutes with the time derivative. Namely
\[
\partial_t \left( \frac{du \otimes du}{ u^2(1-f)^2} \right)
 = \frac{d\partial_t u \otimes du}{ u^2(1-f)^2} +
\frac{ u \otimes d \partial_t u}{ u^2 (1-f)^2} + du \otimes du \,
 \partial_t \left(\frac{1}{u^2 (1-f)^2}\right).
\] Hence
\[
\partial_t w_{ij}= \partial_t \left( \frac{u_i u_j}{ u^2(1-f)^2} \right) = \frac{(\partial_t u)_i u_j}{ u^2(1-f)^2} +
\frac{ u_i (\partial_t u)_j}{ u^2(1-f)^2} + u_i u_j \, \partial_t \left(\frac{1}{u^2 (1-f)^2}\right).
\]  This is identical to the corresponding term in Lemma \ref{Lemma-w(ij)}.
The computation for all other terms are the same also.
\qed

In this section we will work with space time cubes that evolve
with time.  Recall the following notation. Let $(x_0, t_0)$ be a
space time point. For $R>0$ and $T>0$, we write
\[
Q_{R,T}(x_0,t_0) =\{(x, t) \, | \, d(x_0, x, t) < R, \, t_0-T < t
\le t_0 \},
\]

We now prove Theorem \ref{Thm-LogHessRic}.

\proof[Proof of Theorem \ref{Thm-LogHessRic}]
{\bf Part (a).}  Again we set $w=tr (W)$, and
\begin{equation}
\label{operL}
 L= -\partial_{t} + \Delta - \frac{2f}{1-f} \nabla{f} \cdot \nabla.
\end{equation}
According to Lemmas \ref{Lemma-v(ij)RF} and \ref{Lemma-w(ij)RF},
the following equalities hold
\begin{align}
\label{eqVRF}
LV&= (1-f) w V + P, \\
\label{eqWRF}
LW&= 2(1-f)w W +2(V+fW)^{2} +Q,
\end{align}
where $P$ and $Q$ are matrices whose $(i, j)$-th components are given by
\begin{align}\label{eqPRF}\begin{split}
P_{ij}&= \frac{1}{u (1-f)} [- 2
 R_{kijl} u_{kl} + R_{il} u_{jl} + R_{jl} u_{il} ]\\
&=- 2
 R_{kijl} v_{kl} + R_{il} v_{jl} + R_{jl} v_{il} ,
\end{split}\end{align} and
\begin{equation}
\label{eqQRF}Q_{ij}= R_{ik} w_{kj} + R_{jk} w_{ki}.
\end{equation}

For a constant $\alpha \in \mathbb{R}$ to be determined, we have
$$ L(\alpha V + W)= \alpha (1-f) w V + 2(1-f) w W + 2(V+fW)^{2}
+\alpha P + Q.$$ Pick $p \in M$ and a time $t$ where the Ricci
flow is defined. Let $\xi \in T_p M$ be a unit tangent vector at
the point $p$. Under the metric $g(t)$,  we use parallel
translation along geodesics emanating from $p$ to extend $\xi$ to
a smooth vector field in the local coordinate neighborhood. Then
we extend the vector field in time trivially by making it a
constant vector field in time. We still denote the vector field by
$\xi$. Since $V$ and $W$ are $(2, 0)$-tensor fields ($2$-forms),
the function
$$
\lambda = \xi^{T} (\alpha V + W) \xi \equiv (\alpha V + W) (\xi, \xi)
$$
is a well-defined smooth function in a space-time neighborhood of
$(p, t)$.
Since $\xi$ is a parallel vector field  at time $t$,  it holds, at this time $t$,
$$ L \lambda = H +  \xi^{T} (\alpha P + Q) \xi.$$
Here
$$ H= \alpha (1-f) w \xi^{T} V \xi +2 (1-f) w \xi^{T} W \xi + 2 |(V+fW) \xi|^{2}.$$
By $ \alpha \xi^{T} V \xi = \lambda - \xi^{T} W \xi$,
we have
\begin{align*}
H&=  (1-f) w (\lambda - \xi^{T} W \xi) + 2(1-f) w \xi^{T} W \xi + 2 |(V +f W) \xi|^{2} \\
&=  (1-f) w \lambda + (1-f) w \xi^{T} W \xi + 2 |(V+ fW)\xi|^{2}.
\end{align*}

To simplify the last term further, we fix the space time point
$(p, t)$ and assume $\xi$ is chosen as follows. Under the metric
$g(t)$, we let $\xi$ be the time independent vector field
generated via parallel translation through geodesics emanating
from $p$, by an eigenvector of $\alpha V + W$ at $(p, t)$, i.e.,
at $(p, t)$,
$$ (\alpha V + W) \xi =\lambda \xi.$$
 Then by (\ref{H>2lamb2})  in the derivation of (\ref{Llamb}), we find that, if $ \lambda \geq 0$, then at the point $(p, t)$, it
holds, for $\alpha \ge 4$
\[
H \ge \frac{ 2 \lambda^2}{\alpha^2}.
\]Consequently, at time $t$, it holds
\begin{equation}
\label{LlambRF}
 L \lambda \geq \frac{2 \lambda^{2}}{\alpha^{2}} +  \xi^{T} (\alpha P + Q) \xi.
 \end{equation}

Let $\tau$ be a universal constant to be fixed later. With
$\alpha=4$, suppose the $2$-form
$$\alpha V + W-\frac{\tau}{t}g(t)$$ assumes the
eigenvalue $\mu \equiv \lambda- \frac{\tau}{t_1}$ at $ (p_1,
t_1)$, which is the largest for all $x \in M$ and $t \in (0, T]$.
Then $\lambda$ is an eigenvalue of $\alpha V + W$. Under the
metric $g(t_1)$, let $\xi$ be a unit eigenvector of $\alpha V +
W$, which corresponds to $\lambda$.
 Under $g(t_1)$ again, we use
parallel translation along geodesics emanating from $p_1$ to
extend $\xi$ to a smooth vector field in a neighborhood of $p_1$.
We still denote it by $\xi$. Now we regard $\xi$ as a time
independent vector field defined in a space time neighborhood of
$(p_1, t_1)$.
Set
$$ \lambda = \xi^{T} (\alpha V + W) \xi \equiv (\alpha V + W)(\xi, \xi).$$
Then, in the space-time neighborhood where $\xi$ is defined, we
have
\begin{equation}
\label{LLamb=H+}
 L \left(\lambda-\frac{\tau}{t} g(t)(\xi, \xi) \right)  = L
\lambda- \frac{\tau}{t^2} g(t)(\xi, \xi) -\frac{\tau}{t} 2 Ric
(\xi, \xi).
\end{equation}

We now evaluate at $(p_1, t_1)$. Since $\frac{\lambda
-\frac{\tau}{t} g(t)(\xi, \xi)}{g(t)(\xi, \xi)}$ has its
nonnegative maximum at $(p_1, t_1)$, we have,
\begin{equation}
\label{ddlam+dlam}
\Delta
\left(\lambda-\frac{\tau}{t} g(t)(\xi, \xi) \right) \le 0,
\qquad
\nabla \left(\lambda-\frac{\tau}{t} g(t)(\xi, \xi) \right)=0
\end{equation} and
\[
\partial_t \left[ \frac{\left(\alpha V + W -
\frac{\tau}{t} g(t)\right)(\xi, \xi) }{g(t)(\xi, \xi)} \right] \ge
0.
\] Here we point out that even though $\xi$ is a time independent
vector field, its norm changes under $g(t)$. This is the reason
why we need to normalize its norm in the above inequality.
Therefore
\[
  \frac{\partial_t \left[ \left(\alpha V + W -
\frac{\tau}{t} g(t)\right)(\xi, \xi)  \right] }{g(t)(\xi, \xi)} +
\left(\alpha V + W - \frac{\tau}{t} g(t)\right)(\xi, \xi) \frac{2
Ric (\xi, \xi)}{ |g(t)(\xi, \xi)|^2} \ge 0.
\]Since the computation is at $(p_1, x_1)$, this implies
\[
\partial_t (\lambda-\frac{\tau}{t} g(t)(\xi, \xi) ) + (\lambda-\frac{\tau}{t}g(t)(\xi, \xi)) 2
Ric (\xi, \xi) \ge 0.
\]
Substituting this together with (\ref{operL}) and (\ref{ddlam+dlam}) into the left hand side of
(\ref{LLamb=H+}),
we find that
\[
(\lambda-\frac{\tau}{t}) 2
Ric (\xi, \xi) \ge  L
\lambda- \frac{\tau}{t^2} g(t)(\xi, \xi) -\frac{\tau}{t} 2 Ric
(\xi, \xi).
\]
By (\ref{LlambRF}), this induces
$$ \lambda 2 Ric (\xi, \xi)
\geq \frac{2\lambda^{2}}{\alpha^{2}}-\frac{\tau}{t^2}-|\xi^{T}
(\alpha P + Q) \xi | \quad\text{at }(p_1,t_1), $$ or

\begin{equation}
\label{lamb2<}  \frac{2\lambda^{2}}{\alpha^{2}} \le \frac{\tau}{t^2}+|\xi^{T} (\alpha P + Q) \xi |
+ 2 \lambda |Ric(\xi, \xi)|
\quad\text{at }(p_1,t_1).
\end{equation}

To bound $\lambda$ from
above, we need to find an upper bound for $ |\xi^{T} (\alpha P +
Q) \xi | $ at $(p_1, t_1)$. Let $\xi=(\xi_1, ..., \xi_n)$.  By
(\ref{eqPRF}) and (\ref{eqQRF}), we obtain
\begin{align*}
&| \xi^{T}   (\alpha P + Q) \xi |\le | \alpha\xi^{T} P \xi |
+| \xi^{T} Q \xi |\\
&\quad\le \left| \xi_i \xi_j \alpha \left(- 2 R_{kijl} v_{kl} +
R_{il} v_{jl} + R_{jl} v_{il} \right) \right|
+ \left|\xi_i \xi_j ( R_{ik} w_{kj} +R_{jk} w_{ki}) \right| \\
&\quad\le  \left| \xi_i \xi_j \alpha \left[- 2
 R_{kijl} v_{kl} + R_{il} v_{jl} + R_{jl} v_{il}  \right] \right|  + C |Ric| |W|.
\end{align*}
Writing $\alpha v_{kl} = \alpha v_{kl} + w_{kl} - w_{kl}$ etc  in
the last line, we deduce
\be
\al
\label{aP+Q<}
&| \xi^{T}   (\alpha P + Q) \xi |\\
&\quad\le  \left| \xi_i \xi_j  \left[- 2
R_{kijl} (\alpha v_{kl} + w_{kl}) + R_{il} (\alpha v_{jl}  + w_{jl})+ R_{jl} (\alpha v_{il}+ w_{il}) \right] \right|  \\
&\qquad\quad+ \left| \xi_i \xi_j  \left[- 2 R_{kijl}  w_{kl} +
R_{il}  w_{jl}+ R_{jl}  w_{il} \right] \right|  + C |Ric| |W|.
\eal
\ee
At the point $(p_1, t_1)$, we can choose a coordinate system so that the
matrix $\alpha V + W$ is diagonal. Let $\mu_1, \cdots,
\mu_n$ be the eigenvalues of the matrix $\alpha V+W-\frac{\tau}{t} g$, listed
in the increasing order. We claim that the absolute value of $\mu_1$ is bounded from above by $\mu_n$ plus a controlled quantity.  Without loss of generality, we assume
$\mu_1<0$ and $\mu_n=\mu>0$.
Then
$$
\al
| R_{kijl} (\alpha v_{kl} + w_{kl}) | &\le | R_{kijl} (\alpha v_{kl} + w_{kl} - \frac{\tau}{t} \delta_{kl}) |
+ | R_{kijl} \frac{\tau}{t} \delta_{kl} |\\
&\le| \sum^{n}_{k=1}  R_{kijk}
(\alpha v_{kk} + w_{kk}- n \frac{\tau}{t})|  + C |Rm| \frac{\tau}{t} \\
&\le C |Rm| (\mu_n + |\mu_1| +\frac{\tau}{t} ).
\eal
$$
Similarly
$$
|R_{il} (\alpha_{il} v_{jl}  + w_{jl})| \le C |Ric| (\mu_n+
|\mu_1| +\frac{\tau}{t}).
$$
Combining the last three inequalities, we deduce
$$
| \xi^{T}   (\alpha P + Q) \xi | \le C |Rm| (\mu_n
+|\mu_1| + \frac{\tau}{t}) + C  |Rm| |W|.
$$
In the following, we set $K_1=|Rm|_{L^\infty}$. Then
\begin{align*}| \xi^{T}   (\alpha P + Q) \xi | &\le C K_1 (\mu_n +|\mu_1|+\frac{\tau}{t})+ C K_1 |W|\\
&\le C K_1 (\mu_n +|\mu_1|+\frac{\tau}{t}) + C K_1 |W|.
\end{align*}
 Note
\begin{align*} \mu_1 + (n-1) \mu_n &\ge \mu_1+\cdots+\mu_n \\
&= \operatorname{tr} \left(\frac{\alpha
u_{ij}}{u(1-f)}+\frac{u_iu_j}{u^2(1-f)^2} - \frac{\tau}{t} \delta_{ij} \right) \ge \frac{\alpha
\Delta u}{u(1-f)}  - n \frac{\tau}{t}.
\end{align*}
 Hence,
$$|\mu_1| \le (n-1) \mu_n -\frac{\alpha \Delta u}{u(1-f)} + n \frac{\tau}{t}.$$
Then,
$$| \xi^{T}   (\alpha P + Q) \xi |\le C K_1 \left(\mu_n -\frac{\alpha \Delta u}{u(1-f)}
+ \frac{\tau}{t} \right)  + C K_1 |W|.
$$
Now we need a version of the Li-Yau gradient estimate for the heat equation coupled with
the Ricci flow.  Indeed,
by  Theorem 2.7 in \cite{BCP:1}, we have
$$\frac{|\nabla u|^2}{u^2} - 2 \frac{ u_t }{u} \le \frac {c_1}t+ c_2 K_0,$$
where $|Ric|_\infty \le K_0$ for some $K_0\ge 0$. With $u_t =\Delta u$, we
get
$$-\frac{\Delta u }{u} \le \frac {c_1}t+c_2 K_0.$$
Since $0\le u<A$, we have
$$\frac{1}{1-f}\le  1.$$
Therefore, at $(p_1, t_1)$,
$$
| \xi^{T}   (\alpha P + Q) \xi | \le C K_1\left (\mu_n
+\frac{1+\tau}{t} \right)+C K_1K_0 + C K_1 |W|.
$$
By the definition of $W=(w_{ij})$, we have
$$|W|\le \frac{|\nabla u|^2}{u^2(1-f)^2}.$$
By  \cite{Z:1} and also \cite{CaH:1}, we have a curvature independent bound
$$|W|\le C \frac1t.$$
Hence
\begin{align*}
| \xi^{T}   (\alpha P + Q) \xi | &\le C K_1 \left (\mu_n
+\frac{1+\tau}{t}\right)+C(K_1K_0+K_1) \\
&\le  C K_1 \left (\lambda
+\frac{1+\tau}{t}\right)+C(K_1K_0+K_1),
\end{align*}
where we used the relation $\mu_n=\mu =
\lambda-\frac{\tau}{t_1} \le \lambda$.
Substituting this in
(\ref{lamb2<}), we arrive at
$$  \frac{2\lambda^{2}}{\alpha^{2}}\le\frac{\tau}{t^2}
+ C K_1 \left (\lambda +\frac{1+\tau}{t} \right)+C(K_1K_0+K_1) + 2 \lambda K_0
\quad\text{at }(p_1,t_1). $$ A simple application of the Cauchy
inequality yields
$$\frac{\lambda}{\alpha}\le 2 \frac{\sqrt{1+\tau}}{t}+ B\quad\text{at }(p_1,t_1), $$
where $B$ is a nonnegative constant depending only on $K_0, K_1$ and $\tau$, with the property that $B=0$ if $K_0=K_1=0$. Then
$$\lambda-\frac\tau t\leq (2\alpha\sqrt{1+\tau}-\tau)\frac1t+\alpha B\le \alpha B\quad\text{at }(p_1,t_1),$$
by choosing $\tau$ sufficiently large. In fact, for $\alpha=4$, we can take
$\tau=8+2\sqrt{17}$.   Recall that $\mu=\lambda-\frac\tau
t$ at $(p_1, t_1)$ is the largest eigenvalue of the $2$ form
$\alpha V + W - \frac{\tau}{t} g(t)$ on $M \times (0, T]$.
Therefore, given any nonzero tangent vector $\eta \in T_x M$, $x
\in M$, it holds
\[
\frac{\eta^{T} (\alpha V + W) \eta - \frac{\tau}{t} g(t)(\eta, \eta)}{g(t)(\eta, \eta)} \le
\alpha B\quad \text{in} \ M\times (0, T).
\] Thus
$$ \frac{t \eta^{T} (\alpha V + W) \eta}{g(t)(\eta, \eta)} \leq \tau+\alpha Bt.$$
Hence
$$t \frac{ \eta^{T} V \eta }{g(t)(\eta, \eta)} \leq \frac\tau\alpha+Bt.$$
This proves part (a) of the theorem.
\medskip

{\bf Part (b).}
Let $\psi$ be a cutoff function supported in the space-time cube
$Q_{R, T}(x_0, t_0)$ such
that $\psi=1$ in the cube of half the size $Q_{R/2, T/2}(x_0, t_0)$. We also require that
\begin{equation}
\label{psiPrprty} |\nabla \psi| \le \frac{C}{R}, \quad |\Delta
\psi | \le C \frac{K_0+1}{R^2}, \quad \frac{|\partial_t \psi
|}{\sqrt{\psi}} \le C \frac{(K_0+1)}{T}, \quad \frac{|\nabla
\psi|^2}{\psi} \le \frac{C}{R^2}.
\end{equation} Here $K_0$ is a bound on $|Ric|$ as before.
Also the quotients are regarded as $0$ when $\psi=0$ somewhere. It
is well known that such a cutoff function exists. See \cite{BCP:1}
e.g. Without loss of generality we just take $t_0=T$. We also
require that $\psi$ is supported in the slightly shorter space
time cube $Q_{R, 3T/4}(x_0, t_0)$. The cut off function can be
constructed from the distance function. Since we will be
differentiating at a fixed point in space time eventually, we can
use the well know trick by Calabi to get around singular points of
the distance.

Now we consider, for some constant $\beta \in \mathbb{R}^+$ to be determined,
the $2$-form
\[
\psi (\alpha V + W)+ \beta f g(t).
\]
Let $\mu$ be an eigenvalue and $\xi$ be a corresponding eigenvector of $\psi (\alpha V + W)+ \beta f g(t)$ at some point $(p, t)$, then
$$ [\psi (\alpha V + W) + \beta f g(t)] \xi = \mu \xi.$$
Here $g(t) \xi$ stands for the dual vector of the one form $g(t)(\cdot, \xi)$.  In a local coordinate, the above becomes
$$ \psi (\alpha V + W) \xi= (\mu - \beta f) \xi. $$
If $\psi (p, t)\neq 0$, then $\xi$ is also an eigenvector of $\alpha V + W$, corresponding
to the eigenvalue $\lambda$. Here we just define $\lambda$ by
$$ \mu = \psi \lambda + \beta f. $$

As in Part (a), we extend $\xi$ to a time independent vector field in a space-time neighborhood by parallel transport, which is still denoted by $\xi$. Now we extend $\mu$ and $\lambda$
to smooth functions in the same neighborhood by the relation
\[
\mu  = \xi^{T} [ \psi (\alpha V + W) + \beta f g(t)] \xi,
\]
and
\[
\lambda = \xi^{T}  (\alpha V + W)  \xi.
\] Therefore  as functions, $\mu$ and $\lambda$ are also related by
\[
\mu = \psi \lambda + \beta f.
\]We observe that at the point $(p, t)$, $\mu$ and $\lambda$ are eigenvalues of
the respective $2$-forms. However, at different points, this may not be the case.

Following the computation in deriving (\ref{L1psilamb}), we know that
\be
\label{psiL1mu=}
 \psi L_1 \mu =\psi L_1(\psi \lambda + \beta f) = \psi^{2} H +
\psi^2\xi^T(\alpha P+Q)\xi+ \psi \lambda L_1\psi + \beta \psi L_1f.
\ee
Here $L_1$ is the operator given in (\ref{defL1}), $H$ is given by (\ref{eqH}) and $P$ and $Q$ are given by (\ref{eqPRF}) and (\ref{eqQRF}) respectively.

Let $\Omega$ be the parabolic cube given by
$ \Omega= Q_{R,T}(x_0,t_0).$
Then
$$ \mu \big|_{\partial_{p} \Omega} = \beta f \big|_{\partial_{p} \Omega} < 0.  $$
Here $\partial_p$ stands for the parabolic boundary.
We will estimate $\mu$ from above.

Since the term $H$ does not involve time derivative, the inequality (\ref{H>}) in the
fixed metric case still stands. Therefore we have
$$ \psi^{2} H \geq \frac{2}{\alpha^{2}} (\psi \lambda)^{2} + \psi^{2} \lambda
\left(w - \frac{4}{\alpha^{2}} \xi^{T} W \xi\right) - f \psi^{2} \lambda \left(w- \frac{4}{\alpha}\xi ^{T} W \xi\right).$$
At points where $\mu \geq 0$, we have
$$ \psi \lambda + \beta f \geq 0,$$
and hence $ \psi \lambda \geq 0$. Then, for $\alpha \ge 4$, it holds
$$ \psi ^{2} H \geq \frac{2}{\alpha^{2}} (\psi \lambda)^{2}.$$
This, (\ref{L1f>}) and (\ref{psiL1mu=}) yields
\begin{align*}
\psi L_1\mu &\geq  \frac{2}{\alpha^{2}} (\psi \lambda)^{2} +
\psi^2\xi^T(\alpha P+Q)\xi+ \beta \left[ \frac{1}{4} \psi |\nabla f|^{2} - C \frac{|\nabla \psi|^{2}}{\psi}\right] \\
&\qquad - \left[ |\partial_{t} \psi| + |\Delta \psi| + \frac{2 |\nabla \psi|^{2}}{\psi}
+2 \sqrt{\psi} |\nabla f| \cdot\frac{\nabla \psi}{\sqrt{\psi}}\right] \psi \lambda.
\end{align*} As in the previous section, we take
\be
\label{beta=}
\beta=c\sup  \frac{|\nabla \psi|^{2}}{\psi},
\ee with $c$ being sufficiently large. Then
 we can use the Cauchy inequality to prove:
\be
\al
\label{psiL1mu>}
\psi L_1 \mu &\geq  \frac{1}{\alpha^2} (\psi \lambda)^{2}
+\psi^2\xi^T(\alpha P+Q)\xi+ \frac{1}{8} \beta \psi |\nabla f|^{2} \\
&\quad\quad - C \left[ |\partial_{t} \psi|+ |\Delta \psi|
+\frac{|\nabla \psi|^{2}}{\psi} \right]^{2}-
C\sup  \frac{|\nabla \psi|^{4}}{\psi^2}.
\eal
\ee

Let $\mu_{0}$ be a maximal eigenvalue of $\psi (\alpha V+ W) + \beta f g(t)$ in $\Omega$, which associates with a unit eigenvector $\xi$.
Assume $\mu_{0}$ is taken at the space-time point $(p_1, t_1)$.
We are interested only in the case $\mu_{0} \geq 0$.
 Just like in Part (a), under $g(t_1)$, we use
parallel translation along geodesics emanating from $p_1$ to
extend $\xi$ to a smooth vector field in a neighborhood of $p_1$.
We still denote it by $\xi$. Now we regard $\xi$ as a time
independent vector field defined in a space time neighborhood of
$(p_1, t_1)$.
 Set
$$ \mu =  \xi^{T} [\psi (\alpha V+ W) + \beta f g(t)]  \xi.$$
 Since $\frac{\mu}{g(t)(\xi, \xi)}$ has its
nonnegative maximum at $(p_1, t_1)$, we have,  at this point,
\begin{equation}
\label{ddmu+dmu}
\Delta
\mu = \Delta (\mu/g(t)(\xi, \xi)) \le 0,
\qquad
\nabla \mu= \nabla (\mu/g(t)(\xi, \xi))=0,
\end{equation} and
\[
\partial_t \left[ \frac{\mu}{g(t)(\xi, \xi)} \right] \ge
0.
\] Therefore
\[
  \frac{\partial_t \mu}{g(t)(\xi, \xi)} +
\mu \frac{2
Ric (\xi, \xi)}{ |g(t)(\xi, \xi)|^2} \ge 0.
\]Since the computation is at $(p_1, x_1)$, this implies
\[
\partial_t \mu + \mu 2
Ric (\xi, \xi) \ge 0.
\]
Recall from (\ref{defL1}) that
$$L_1 = - \partial_{t} + \Delta - \frac{2 f}{1-f} \nabla f \cdot \nabla - \frac{2 \nabla \psi}{\psi} \ \nabla.$$ Hence we can plug the above inequality and (\ref{ddmu+dmu}) in (\ref{psiL1mu>}) to deduce
\begin{align*}
\mu \psi 2 Ric (\xi, \xi) \geq \psi L_1 \mu &\geq \frac{1}{\alpha^{2}} (\psi \lambda)^{2} +\psi^2\xi^T(\alpha P+Q)\xi\\
&\qquad- C \left[ |\partial_{t} \psi| + |\Delta \psi| +
\frac{|\nabla \psi|^{2}}{\psi} \right]^{2}- C\sup  \frac{|\nabla
\psi|^{4}}{\psi^2}.\end{align*}
This implies, at $(p_1, t_1)$, \be
\label{(psilamb)2<}
 \frac{1}{\alpha^{2}} (\psi \lambda)^{2}
 \le\psi^2|\xi^T(\alpha P+Q)\xi|+C \left(\frac{1}{T} + \frac{1}{R^{2}} \right)^2
  + \psi \mu 2 K_0.
\ee Now we control the right hand side. From (\ref{aP+Q<}) in Part
(a), we have
\[
\al
&\psi | \xi^{T}   (\alpha P + Q) \xi |\\
&\quad\le   \left| \xi_i \xi_j  \left[- 2
R_{kijl} \psi (\alpha v_{kl} + w_{kl}) + R_{il} \psi (\alpha v_{jl}  + w_{jl})+ R_{jl} \psi (\alpha v_{il}+ w_{il}) \right] \right|  \\
&\qquad\quad+ \psi \left| \xi_i \xi_j  \left[- 2 R_{kijl}  w_{kl} +
R_{il}  w_{jl}+ R_{jl}  w_{il} \right] \right|  + C\psi |Ric| |W| \\
&\le \left| \xi_i \xi_j  \left[- 2
R_{kijl} \psi(\alpha v_{kl} + w_{kl}) + R_{il} \psi(\alpha v_{jl}  + w_{jl})+ R_{jl} \psi (\alpha v_{il}+ w_{il}) \right] \right|\\
&\qquad + C (K_0+K_1) \psi |W| \\
&\le \big| \xi_i \xi_j  \big[- 2
R_{kijl} \{ \psi (\alpha v_{kl} + w_{kl})- \beta f \delta_{kl}\} + R_{il} \{\psi (\alpha v_{jl}  + w_{jl}- \beta f \delta_{kl}\}\\
&\qquad \qquad + R_{jl} \{\psi (\alpha v_{il}+ w_{il} - \beta f \delta_{kl}\} \big] \big| \\
& \qquad +  C (K_0+K_1) \beta \psi |f|+ C (K_0+K_1) \psi |W|.
\eal
\] Let $\mu_1, ..., \mu_n$ be the eigenvalues of the $2$-form $\psi(\alpha V + W) + \beta f g(t)$
at the space-time point $(p_1, t_1)$, which are listed in increasing order. We assume without loss
of generality that $\mu_1<0$. Then the above inequality implies
\[
\psi | \xi^{T}   (\alpha P + Q) \xi | \le C (K_0 + K_1) (\mu_n + |\mu_1|) +
 C (K_0+K_1) \beta  |f|+ C (K_0+K_1) \psi |W|.
\]
Observe that
\begin{align*} \mu_1 + (n-1) \mu_n &\ge \mu_1+\cdots+\mu_n \\
&= \operatorname{tr}\left[  \psi\left(\frac{\alpha
u_{ij}}{u(1-f)}+\frac{u_iu_j}{u^2(1-f)^2} \right) + \beta f \delta_{ij} \right] \ge \psi \frac{\alpha
\Delta u}{u(1-f)}  + n \beta f .
\end{align*}
 Hence,
$$|\mu_1| \le (n-1) \mu_n -\psi \frac{\alpha \Delta u}{u(1-f)} + n \beta |f|.$$
Therefore,
\be
\label{psi(aP+Q)<} \psi | \xi^{T}   (\alpha P + Q)
\xi |\le C K_1 \left(\mu_n - \psi \frac{\alpha \Delta u}{u(1-f)}
 \right) + C K_1 \psi |W| + C K_1 \beta |f|.
\ee
 By  Theorem 2.7 in \cite{BCP:1} again, we have
$$
 \psi
\left(\frac{|\nabla u|^2}{u^2} - 2 \frac{ u_t }{u} \right)
 \le c_1 \left(\frac {1}{1+t} + \frac{1}{R^2} \right)+ c_2 K_0,
 $$
where $K_0=|Ric|_\infty$. We mention that this inequality is not
exactly the one stated in their Theorem 2.7, due to the appearance
of the cutoff function $\psi$ in the front and the appearance of
$\frac {1}{1+t}$ instead of $\frac1t$. However, this is actually
what was first proved at the first line in p3532 there. This is
also the case for the Li-Yau inequality in the fixed metric case.
Since, by construction, $\psi$ is supported in the slightly
shorter space time cube $Q_{R, 3T/4}(x_0, t_0)$ and $t_0=T$, we
have
\[
\psi \left(\frac{|\nabla u|^2}{u^2} - 2 \frac{ u_t }{u} \right)
 \le c_3 \left(\frac {1}{T} + \frac{1}{R^2} \right)+ c_2 K_0.
\]
Using $u_t =\Delta u$, we get
$$-\psi \frac{\Delta u }{u} \le c_3 \left(\frac {1}{T} + \frac{1}{R^2} \right)+ c_2 K_0.$$
Since $0\le u<A$, we have
$$\frac{1}{1-f}\le  1.$$
Substituting these in (\ref{psi(aP+Q)<}), we find that, at $(p_1,
t_1)$, it holds
$$
\psi | \xi^{T}   (\alpha P + Q) \xi | \le C K_1\left (\lambda +c_3
\left(\frac {1}{T} + \frac{1}{R^2} \right)+ c_2 K_0\right) + C K_1
\psi |W| + C K_1 \beta |f|.
$$ Here we also used the inequality $\mu_n=\mu =\psi \lambda +
\beta f \le \psi \lambda$. By the definition of $W=(w_{ij})$, we
have
$$|W|\le \frac{|\nabla u|^2}{u^2(1-f)^2}.$$
According to Theorem 2.2 in \cite{BCP:1},
we obtain
$$\psi |W|\le C \left( \frac{1}{t+1} + \frac{1}{R^2} +K_0 \right) \le
C \left( \frac{1}{T} + \frac{1}{R^2} +K_0 \right) .$$ Again the
cutoff function $\psi$ and term $\frac{1}{t+1}$ were not in the
original statement. But this was proven in that paper on the way
to prove Theorem 2.2 there. Also the last inequality follows from
the choice of the support for $\psi$. Now we know that
$$
\psi^2 | \xi^{T}   (\alpha P + Q) \xi |\le C K_1 \psi \lambda + C
K_1 \frac{1}{R^2} + C (K_1 +1) K_0 + C K_1 \beta \psi |f|.
$$ Substituting this in (\ref{(psilamb)2<}), we conclude, at $(p_1, t_1)$,
\begin{align*} \frac{1}{\alpha^{2}}
(\psi \lambda)^{2}&\le C K_1\psi \lambda +CK_1 K_0 + C K_1
\left(1+\frac{1}{T}+\frac{1}{R^2}\right)
+C\left(\frac1T+\frac{1}{R^2}\right)^2\\
&\qquad+C K_1 \beta \psi |f| + 2K_0 \psi \mu.
\end{align*} As $\mu = \psi \lambda + \beta f \le \psi \lambda$ and $K_0 \le C K_1$, this yields
$$\psi \lambda\le C \left(\frac1T+\frac{1}{R^2}+B\right)+C \sqrt{K_1 \beta \psi |f|},$$
where $B$ is a nonnegative constant depending only on $K_0, K_1$
such that $B=0$ when $K_0=K_1=0$. This shows
\[
\mu = \psi \lambda + \beta f \le C
\left(\frac1T+\frac{1}{R^2}+B\right)+C \sqrt{K_1 \beta \psi |f|} +
\beta f.
\]

Since $f<0$, we know that $C \sqrt{K_1 \beta \psi |f|} + \beta f
\le C K_1$, which implies, at $(p_1, t_1)$, that
\[
\mu = \psi \lambda + \beta f \le C
\left(\frac1T+\frac{1}{R^2}+B\right).
\]
Recall $\mu$ at $(p_1, t_1)$ is the maximum of the eigenvalues of
the $2$-form $\psi(\alpha V + W) + \beta f$ in the cube $Q_{R,
T}(x_0, t_0)$. Hence for any
 unit tangent vector $\xi$ at a space-time point $(x, t) \in Q_{R, T}(x_0,
t_0)$, it holds
$$ \psi \xi^{T} (\alpha V + W) \xi + \beta f \leq
C \left( \frac{1}{T} + \frac{1}{R^{2}} +B\right), \quad
$$ or
$$ \psi \xi^{T} (\alpha V + W) \xi \leq
C \left(\frac{1}{T} + \frac{1}{R^{2}} +B\right)+ \beta |f|.
$$
With the choice of $\beta = c\sup  \frac{|\nabla
\psi|^{2}}{\psi}=\frac{c C}{R^2}$ in (\ref{beta=}) with $c$
sufficiently large, we then have
$$ \psi \xi^{T} V \xi \leq C \left( \frac{1}{T} + \frac{1}{R^{2}} +B\right) (1-f),$$
and hence
$$ \psi \frac{u_{ij} \xi_{i} \xi_{j}}{u} \leq C \left( \frac{1}{T} + \frac{1}{R^{2}} +B\right) (1-f)^{2}.$$
This implies the desired estimate.

\qed

\end{document}